\documentclass[11pt,a4paper]{amsart}

\usepackage{graphics, amsmath, amsfonts, amssymb, amsthm, amscd, mathrsfs, amsrefs, color}
\input xy

\usepackage[mathscr]{eucal}
 \usepackage{extarrows}

\theoremstyle{plain}
\newtheorem{Thm}{Theorem}[section]
\newtheorem{Cor}[Thm]{Corollary}
\newtheorem{Lem}[Thm]{Lemma}
\newtheorem{Prop}[Thm]{Proposition}

\theoremstyle{definition}

\theoremstyle{remark}
\newtheorem{Rem}[Thm]{Remark}

\newcommand{\N}{\ensuremath{\mathbb{N}}}
\newcommand{\C}{\ensuremath{\mathbb{C}}}
\newcommand{\PP}{\ensuremath{\mathbb{P}}}
\newcommand{\cO}{{\ensuremath{\mathcal{O}}}}

\numberwithin{equation}{section}

\begin{document}

\title[Embedding  open Riemann  surfaces  with isolated punctures into $\C^2$] 
{Embedding  Riemann  surfaces  with isolated punctures into  the complex plane }
\author{Frank Kutzschebauch}
\address{Departement Mathematik\\
Universit\"at Bern\\
Sidlerstrasse 5, CH--3012 Bern, Switzerland}
\email{frank.kutzschebauch@math.unibe.ch}
\author{Pierre-Marie Poloni}
\address{Universit\"at Basel\\
Departement Mathematik und Informatik\\
Spiegelgasse 1\\ CH--4051 Basel\\ Switzerland}
\email{pierre-marie.poloni@unibas.ch}

\thanks{Kutzschebauch partially supported by Schweizerischer Nationalfonds Grant 200021-116165 }
\date{\today}
\bibliographystyle{amsalpha}
\begin{abstract} 
We enlarge the class of open Riemann surfaces known to be holomorphically  embeddable into the plane by allowing them to have additional isolated punctures  compared to the  known embedding results.
\end{abstract}

\maketitle


\section{Introduction} The problem   whether every open Riemann surface embeds properly holomorphically into $\C^2$ is notoriously difficult and has attracted a lot of attention in the past decades.  This is  the last remaining case  of Otto Forsters conjecture about the embedding dimension
of Stein manifolds, since Gromov, Eliashberg and Sch\"urmann proved that any Stein manifold of dimension $n$ admits a proper holomorphic embedding into $\C^N$ for $N= \left[3n/2\right]+1$. Their proof uses the Oka principle which does not apply in the case of Riemann surfaces in $\C^2$.  For a history of the problem and known results for dimension $n=1$ we refer to the monograph of Forstneri\v c \cite{For}. 

An important class of open Riemann surfaces are surfaces  $X= \hat X \setminus C$ obtained after removing a closed set $C$ from a compact Riemann
surface $ \hat X$. If $C$ consists only of "holes", i.e.~~of open sets bounded by smooth curves, we talk about bordered Riemann surfaces. A breakthrough has been obtained by Forn\ae ss Wold  in his dissertation  (see \cite{Wold1, Wold2, Wold3}) which gives strong tools  in that case. For example, it was recently shown in \cite{FW2013} that every domain in the Riemann sphere with at most countably many boundary components, none of which are points, admits a proper holomorphic embedding into $\C^2$. 

 On the other hand, not many methods are known how to deal with isolated points in the boundary.  In particular, examples of the form  $X= \C \setminus \{    \cup_ {n\in\mathbb{N}} \{ a_n, b_n \}\cup \{0\}\}$, where $a_n$ and $b_n$ denote two sequences of non-zero complex numbers converging to infinity and $0$ respectively, were   promoted  by Globevnik  and Forn\ae ss as the easiest examples   not known to be holomorphically embeddable into $\C^2$. Also  it was believed  in the Several Complex Variables  community to be difficult to embed an elliptic curve (torus) with finitely many points removed properly holomorphically into $\C^2$. Surprisingly enough, it had been overlooked (see for example the overview article \cite{BN}) that Sathaye had solved this problem algebraically back in 1977 \cite{Sat}. 
 
The present  paper grew out of an attempt to understand Sathaye's result and proof  over the field  of complex numbers in a more geometric way. 
Our main result is the following generalization of Sathaye's result together with the solution to the  problem  promoted  by Globevnik  and Forn\ae ss. 

\begin{Thm}\label{main} The following open Riemann surfaces admit a proper holomorphic embedding into $\C^2$:

\begin{itemize} 

\item the Riemann sphere with a (nonempty) countable closed subset  containing  at most 2 accumulation points  removed,

\item any compact Riemann surface of genus $1$ (torus) with  a (nonempty) closed countable set containing at most one accumulation point removed,

\item any hyperelliptic Riemann surface with a countable closed set $C$ removed with the properties that $C$ contains a fibre $F=R^{-1} (p)$ (consisting either of two points or of a single Weierstrass point) of the Riemann map $R$ and  that  all accumulation points of $C$ are contained in that fibre  $F$.

\end{itemize}
The same holds if $X$ is as above with  additionally a finite number of smoothly bounded regions removed.
\end{Thm}

Note that when the removed set is a finite set, the second and third cases  correspond to the theorem of Sathaye \cite{Sat} and we give a new proof of it.

We thank J\'er\'emy Blanc for inspiring discussions on the subject.
\section{proofs}

\begin{Lem} \label{local}
Let $\phi$ be an embedding of the disc $\Delta$ as the  graph $ \phi (x) = (x, f(x))$ of a holomorphic map $f: \Delta \to \C$. Let $ \alpha : \C^2 \to \C^2$ be a birational map of $\C^2$ of the form  $\alpha (x, y) = (x, \frac{ y-a(x)} {b(x)}) $ with the property that  $a(0) = f(0)$ and $ b$ has no zero's except  one simple zero at $0$ . Then $\alpha \circ \phi$  is again an embedding as a graph.

\end{Lem}

\begin{proof} Outside $x=0$ the map $\alpha \circ \phi$  is the graph of $ g(x) =  \frac { f(x)-a(x)} {b(x)}$ which is holomorphic there.  At $x=0$, we know by L'Hopital rule that $g(x)$  converges to 
$\frac {f^\prime (0)-a^\prime (0) }  {b^\prime (0)}$, 
which  is finite by the assumption that $b$ has a simple zero.  Hence, we get a holomorphic extension of $g$ to $0$.
\end{proof}

\begin{Prop} \label{general} Given a Riemann surface $X$ embedded into $\C^2$, denote by $\pi_X$  the restriction  to $X$ of the projection $\pi: \C^2 \to \C$ onto the first coordinate and suppose that $\pi_X$ is a proper map. Let $P\subset X$ be a closed discrete subset of $X$. Suppose that if a  fiber  $\pi_X^{-1}  (x) $  contains some points of $P$, then one of the following two cases holds: either  $\pi_X^{-1}  (x) $  consists only in points of $P$, or  all but exactly one point of  $ \pi^{-1}  (x)$ are in $P$ and  $\pi _X $ is moreover a submersion at the single point in $ \pi^{-1}  (x) \setminus P$. Then, $X\setminus P$
admits a proper holomorphic embedding into $\C^2$.
\end{Prop}

\begin{proof}
Let $x_i$, $i = 1, 2, \ldots$, denote the first coordinates of the   points of $P$. Since $\pi_X$ is proper, this is a discrete (possibly finite) subset of $\C$. By Weierstrass theorem we can thus construct a  holomorphic function $b \in \cO (\C) $ with simple zero's at the points $x_i$ (and no other zero's). 
For each $i$, we choose a complex number $y_i$ as follows. If the set $\pi_X^{-1} (x_i) \setminus P$ is not empty, then we let $y_i$ be equal to the second coordinate of the unique point of $\pi_X^{-1} (x_i) \setminus P$. Otherwise, we let $y_i$ be any complex number different from all second coordinates of  points from $\pi_X^{-1} (x_i)$.  By Mittag-Leffler approximation theorem we can construct a holomorphic function $a \in \cO (\C)$ with $a(x_i) = y_i$. Now let $ \alpha : \C^2 \to \C^2$ be the birational map of $\C^2$ defined by  $\alpha (x, y) = (x, \frac{ y-a(x)} {b(x)})$. Since $\alpha$ is a bimeromorphic  map which is defined on $X\setminus P$,
the restriction of $\alpha$ to $X\setminus P$ is injective. By Lemma \ref{local} the restriction of $\alpha$ to a neighbourhood of a point in $\pi_X^{-1} (x_i) \setminus P$ is an injective holomorphic immersion.  By construction, the restriction of $\alpha$ to a neighbourhood of a point in $p \in P$ is the graph of a meromorphic map with a single pole in $p$. This, together with the properness of $\pi_X$,  implies that the restriction of $\alpha$ to $X\setminus P$
is proper, giving thus the desired proper holomorphic embedding into $\C^2$.
\end{proof}

\begin{Cor} \label{1} The Riemann sphere with a (nonempty) countable closed subset containing at most two accumulation points  removed 
admits a proper holomorphic embedding into $\C^2$.
\end{Cor}

\begin{proof} Let $X = \PP^1 (\C) \setminus C$, where $C$ is countable closed with at most two accumulation points.
If there are no accumulation points, i.e.~if $C= \{c_1, c_2, \ldots , c_n\}$ is a finite set, we can assume that  one of the points, say $c_1$, is  at infinity and then embed $ \PP^1 (\C) \setminus C = \C \setminus  \{c_2, c_3, \ldots , c_n\}$ as the graph of the meromorphic function $\prod_{i=2}^n \frac 1 {x-c_i}$. If there is one accumulation point, we can assume that this point is at infinity and that $X$ is simply the complement in $\C$ of a sequence of points converging to infinity. It is then easy to embed $X$ as the graph of a meromorphic function with poles at the points of the removed sequence. 

If there are two accumulation points, we can first apply an automorphism of $\PP^1 (\C)$   that sends these two points   to $\infty$
and $0$, respectively. Considering $\C^\star$  embedded into $\C^2$ by $x \mapsto (x, \frac 1 x )$,  our   aim is now  to embed  $\C^\star \setminus P$, where $P=C\setminus\{0, \infty\}$ is a discrete closed subset into $\C^2$. By a suitable linear change of coordinates,    we can obtain $\C^\star$ as being the zero set of $y^2 = x^2-1$ in $\C^2$. The projection to the first coordinate is then a proper
$2:1$ ramified covering and  together with the discrete set $P$   fulfils the conditions of Proposition \ref{general} whose application finishes the proof.
\end{proof}

\begin{Cor} \label{2}  Any compact Riemann surface of genus $1$ (torus) with a countable closed  set  containing  at most one accumulation point removed admits a proper holomorphic embedding into $\C^2$.
\end{Cor}

\begin{proof} Any compact Riemann surface of genus $1$ can be written (in the Weierstrass normal form) as the Riemann surface of the square root of a polynomial with three distinct zero's. This means that it is the compactification
(by one point at infinity) of the affine submanifold of $\C^2$ given by
$$ y^2 = x (x-1) (x-A)$$
where $A$ is any complex number distinct from $0$ and $1$.  Since the automorphism group of a torus acts transitively we can, in case of one accumulation point, put that  point to $\infty$  and in case of removing a finite set put one of these finitely many points  to $\infty$. Again to remove the remaining  discrete set, we can apply Proposition \ref{general} to the projection to the $x$-coordinate.
\end{proof}

Recall that a hyperelliptic Riemann surface $\hat X$ is the compactification of  the square root of a polynomial with $n \ge 4$ distinct zero's.  It is compactified by two points if $n$ is even and  by one single point at infinity otherwise. It is the class of compact Riemann surfaces which admit a (unique up to automorphisms of $\PP^1 (\C)$) $2:1$ ramified covering $R: \hat X \to \PP^1 (\C)$ over the Riemann sphere $\PP^1 (\C)$. If the surface is of genus $g$, then 
 the number of ramification points is $2g+2$ by the Riemann-Hurwitz formula.  These points are exactly the Weierstrass
 points of $\hat X$.
 
 \begin{Cor} \label{3} Any hyperelliptic Riemann surface with a countable closed set $C$ removed  with the properties that $C$ contains a fibre $F=R^{-1} (p)$ of the Riemann map $R$ (consisting either of two points or a single Weierstrass point)  and  that  all accumulation points of $C$ are contained in that fibre  $F$
 admits a proper holomorphic embedding into $\C^2$.
 \end{Cor}
 
 \begin{proof} By the transitivity of the automorphism group of $\PP^1 (\C)$ we can assume that $F$ is at infinity
 and the problem is to embed an affine curve given by 
 $$y^2 = x (x-1) (x-A_1) \cdots (x-A_N)$$ 
 with a closed discrete set $P$ removed. Proposition \ref{general} does the job.
 
\end{proof}

Exactly the same proof implies.

 \begin{Cor} \label{4} Let $f \in \cO (\C)$ be any holomorphic function with  (countably many) simple roots. Then the Stein manifold  $x^2 = f (y)$ (of infinite genus) with
 any  closed discrete subset removed has a proper holomorphic embedding into $\C^2$.
 \end{Cor}

 Proof of Theorem \ref{main}: The three cases in our main theorem correspond to Corollaries \ref{1}, \ref{2}, \ref{3}. The last assertion is a consequence of the following result of Forstneri\v c and Wold \cite[Corollary 1.2]{FW}.
 
 \begin{Thm} Assume that $X$ is a compact bordered Riemann surface with boundary of class $C^r$ for some $r > 1$. If $f: X \hookrightarrow \C^2$
  is a $C^1$ embedding (not necessarily proper) that is holomorphic in the interior $X^0 = X \setminus  \partial X$, then $f$ can be approximated uniformly on compacts in $X^0$ by proper holomorphic embeddings $X^0 \hookrightarrow  \C^2$.
 \end{Thm}
 
 \begin{Rem}Combining the proof of our main theorem with the results in \cite{KLW} (see Theorem 2) one can add
 interpolation condition on discrete sets. More precisely: let $X$ be any open Riemann surface as in Theorem \ref{main} or Corollary \ref{4}. If $\{ a_i\}_{i=1}^\infty$ is any discrete sequence of points in $X$ and $\{ b_i\}_{i=1}^\infty$ any discrete 
 subset in $\C^2$, then the proper holomorphic embedding $\varphi: X \hookrightarrow \C^2$ can be chosen to satisfy $$\varphi (a_i) = b_i  \ \  \forall \  i \in \N.$$
 \end{Rem}

\end{document}